\documentclass{amsart}

\usepackage{amssymb}

\theoremstyle{definition}

\theoremstyle{remark}

\numberwithin{equation}{section}
\numberwithin{equation}{section} 

\newtheorem{thm}{Theorem}[section]
\newtheorem{cor}[thm]{Corollary}

\newtheorem{prop}[thm]{Proposition}
\theoremstyle{definition}

\theoremstyle{remark}

\title{} \author{} \date{}
\begin{document}

\title{} \author{} \date{}
\markboth{Marko Kosti\'c, Daniel Velinov}{Vector-valued almost automorphic distributions, ...}\setcounter{page}{1}
\thispagestyle{empty}

\vspace{0.8cm}
\title[VECTOR-VALUED ALMOST AUTOMORPHIC (ULTRA)DISTRIBUTIONS]{VECTOR-VALUED ALMOST AUTOMORPHIC DISTRIBUTIONS AND VECTOR-VALUED ALMOST AUTOMORPHIC ULTRADISTRIBUTIONS}

\author{Marko Kosti\' c}
\address{Faculty of Technical Sciences,
University of Novi Sad,
Trg D. Obradovi\' ca 6, 21125 Novi Sad, Serbia}
\email{marco.s@verat.net}

\author{Daniel Velinov}
\address{Department for Mathematics, Faculty of Civil Engineering, Ss. Cyril and Methodius University, Skopje,
Partizanski Odredi
24, P.O. box 560, 1000 Skopje, Macedonia}
\email{velinovd@gf.ukim.edu.mk}


\begin{abstract}
In this paper, we introduce the notions of  a vector-valued almost automorphic distribution and a vector-valued almost automorphic ultradistribution, working in the framework of complex Banach spaces. We prove several structural characterizations for the introduced classes.
\\[2mm] {\it AMS Mathematics Subject Classification $(2010)$}: 46F05, 42A75, 11K70, 34A27, 35B15
\\[1mm] {\it Key words and phrases:} vector-valued almost automorphic distributions, vector-valued almost automorphic ultradistributions,
Banach spaces

\end{abstract}
\maketitle

\section{Introduction and preliminaries}

The notion
of a scalar-valued  almost  automorphic  function was introduced by S. Bochner
\cite{bochner} in 1962. The
first  systematic
study of
almost  automorphic functions
on
topological groups was carried out by
W. A. Veech \cite{veech}-\cite{veech-prim} during the period 1965-1967. The reader may consult the monographs \cite{diagana} by T. Diagana, \cite{gaston} by G. M. N'Gu\' er\' ekata and \cite{nova-mono} by M. Kosti\' c  for the basic information about almost automorphic functions, asymptotically almost automorphic functions, their generalizations and various applications to abstract integro-differential equations in Banach spaces.

The notion of a scalar-valued almost automorphic distribution was introduced by C. Bouzar and Z. Tchouar \cite{buzar-tress} in 2017, while the notion of a scalar-valued almost automorphic Colombeau generalized function was introduced by C. Bouzar, M. T. Khalladi and F. Z. Tchouar \cite{buzar-tres} in 2015 (see also the pioneering researches of I. Cioranescu \cite{ioana}-\cite{ioana-tres} and M. C. G\' omez-Collado \cite{gomezinjo} for almost periodic classes). As mentioned in the abstract, the main aim of this paper is to introduce the notions of  a vector-valued almost automorphic distribution and a vector-valued almost automorphic ultradistribution in a complex Banach space. We provide several structural profilations for the introduced classes.

The organization and main ideas of paper are given as follows. In Subsection 1.1, we remind ourselves of the elementary facts about Komatsu's approach to vector-valued ultradistributions. Section 2 is written in an expository manner, and its aim is to transfer the results of C. Bouzar and Z. Tchouar \cite{buzar-tress} to vector-valued case. In Section 3, we introduce the notion of a vector-valued almost automorphic ultradistribution and further analyze this concept. In such a way, we continue the research study of vector-valued almost periodic ultradistributions, carried out recently by the first named author \cite{mat-bilten}.

We use the standard notation throughout the paper. By $(X,\|\cdot \|)$ we denote a complex Banach space. The symbol $C_{b}({\mathbb R} : X)$ and $C(K:X),$ where $K$ is a non-empty compact subset of ${\mathbb R},$ stand for the spaces consisting of all bounded continuous functions ${\mathbb R} \mapsto X$ and all continuous functions $K\mapsto X,$ respectively. Both spaces are endowed with sup-norm.
Let $f : {\mathbb R} \rightarrow X$ be continuous. Then we say that
$f(\cdot)$ is almost automorphic, a.a. for short, iff for every real sequence $(b_{n})$ there exist a subsequence $(a_{n})$ of $(b_{n})$ and a map $g : {\mathbb R} \rightarrow X$ such that
\begin{align*}
\lim_{n\rightarrow \infty}f\bigl( t+a_{n}\bigr)=g(t)\ \mbox{ and } \  \lim_{n\rightarrow \infty}g\bigl( t-a_{n}\bigr)=f(t),
\end{align*}
pointwise for $t\in {\mathbb R}.$ If this is the case, we have that $f\in C_{b}({\mathbb R} : X)$ and the limit function $g(\cdot)$ is bounded on ${\mathbb R}$ but not necessarily continuous on ${\mathbb R}.$ The vector space consisting of all almost automorphic functions is denoted by $AA({\mathbb R} :X).$ Owing to Bochner's criterion, any almost periodic function has to be almost automorphic; the converse statement is not true, however \cite{diagana}.

In this paper, we will use the following notion of a Stepanov $p$-almost automorphic function (see e.g. the paper \cite{fatajou} by S. Fatajou, N. Van Minh, G.\,M. N'Gu\'er\'ekata and A. Pankov): Let $1\leq p<\infty.$ A function $f\in L_{loc}^{p}({\mathbb R}:X)$ is said to be Stepanov $p$-almost automorphic, $S^{p}$-almost automorphic or $S^{p}$-a.a. shortly, iff for
every real sequence $(a_{n}),$ there exists a subsequence $(a_{n_{k}})$
and a function $g\in L_{loc}^{p}({\mathbb R}:X)$ such that
\begin{align*}
\lim_{k\rightarrow \infty}\int^{t+1}_{t}\Bigl \| f\bigl(a_{n_{k}}+s\bigr) -g(s)\Bigr \|^{p} \, ds =0
\end{align*}
and
\begin{align*}
\lim_{k\rightarrow \infty}\int^{t+1}_{t}\Bigl \| g\bigl( s-a_{n_{k}}\bigr) -f(s)\Bigr \|^{p} \, ds =0
\end{align*}
for each $ t\in {\mathbb R}.$ It is checked at once that the $S^{p}$-almost automorphy of $f(\cdot)$ implies the almost automorphy of the mapping
$
\hat{f} : {\mathbb R} \rightarrow L^{p}([0,1] :X)
$ defined by $\hat{f}(t):=f(t+\cdot),$ $t\in {\mathbb R},$ with the limit function being $g(\cdot)(s):=g(s+\cdot)$ for a.e. $s\in [0,1],$ so that any $S^{p}$-almost automorphic function $f(\cdot)$ has to be $S^{p}$-bounded ($1\leq p<\infty$); see \cite{nova-mono} for the notion. The vector space consisting of all $S^{p}$-almost automorphic functions is denoted by $ AAS^{p}({\mathbb R} : X).$
If $1\leq p<q<\infty$ and $f(\cdot)$ is Stepanov $q$-almost automorphic, then $f(\cdot)$ is Stepanov $p$-almost automorphic.
If $f(\cdot)$ is an almost automorphic function,
then $f(\cdot)$ is $S^p$-almost automorphic, for any $p \in [1,\infty).$ The converse statement is false, however.

\subsection{Vector-valued ultradistributions}\label{vvult}

There are a great number of different approaches to the theory of ultradistributions. For the sake of brevity, in this paper we will always follow Komatsu's approach, with
the sequence $(M_p)$ of positive real numbers
satisfying $M_0=1$ and the following conditions:
(M.1): $M_p^2\leq M_{p+1} M_{p-1},\;\;p\in\mathbb{N},$
(M.2): $
M_p\leq AH^p\sup_{0\leq i\leq p}M_iM_{p-i},\;\;p\in\mathbb{N},$ for some $A,\ H>1,$
(M.3'): $
\sum_{p=1}^{\infty}\frac{M_{p-1}}{M_p}<\infty .
$
Any use of the condition\\ (M.3):
$\sup_{p\in\mathbb{N}}\sum_{q=p+1}^{\infty}\frac{M_{q-1}M_{p+1}}{pM_pM_q}<\infty,$
which is slightly stronger than (M.3$'$), will be explicitly accented.

It is well-known that
the Gevrey sequence  $(p!^s)$ satisfies the above conditions ($s>1$). Define
$m_p:=\frac{M_p}{M_{p-1}}$, $p\in\mathbb{N}.$

The space of Beurling,
resp., Roumieu ultradifferentiable functions, is
defined by
$
\mathcal{D}^{(M_p)}
:=\text{indlim}_{K\Subset\Subset\mathbb{R}}\mathcal{D}^{(M_p)}_{K},
$
resp.,
$
\mathcal{D}^{\{M_p\}}
:=\text{indlim}_{K\Subset\Subset\mathbb{R}}\mathcal{D}^{\{M_p\}}_{K},
$
where
$
\mathcal{D}^{(M_p)}_K:=\text{projlim}_{h\to\infty}\mathcal{D}^{M_p,h}_{K},$
resp.,
$\mathcal{D}^{\{M_p\}}_K:=\text{indlim}_{h\to 0}\mathcal{D}^{M_p,h}_{K},
$
\begin{align*}
\mathcal{D}^{M_p,h}_K:=\bigl\{\phi\in C^\infty(\mathbb{R}): \text{supp} \phi\subseteq K,\;\|\phi\|_{M_p,h,K}<\infty\bigr\}
\end{align*}
and
\begin{align*}
\|\phi\|_{M_p,h,K}:=\sup\Biggl\{\frac{h^p|\phi^{(p)}(t)|}{M_p} : t\in K,\;p\in\mathbb{N}_0\Biggr\}.
\end{align*}
In the sequel, the asterisk $*$ is used to denote both, the Beurling case $(M_p)$ or
the Roumieu case $\{M_p\}$.
The space
consisted of all continuous linear functions from
$\mathcal{D}^*$ into $X,$ denoted by
$\mathcal{D}^{\prime *}(X):=L(\mathcal{D}^*:X),$
is said to be the space of all $X$-valued ultradistributions of $\ast$-class. We also need the notion of space
$\mathcal{E}^*(X),$ defined as  $
\mathcal{E}^{\ast}(X)
:=\text{indlim}_{K\Subset\Subset\mathbb{R}}\mathcal{E}^{\ast}_{K}(X),
$
where in Beurling case
$
\mathcal{E}^{(M_p)}_{K}(X):=\text{projlim}_{h\to\infty}\mathcal{E}^{M_p,h}_{K}(X),$
resp., in Roumieu case
$\mathcal{E}^{\{M_p\}}_{K}(X):=\text{indlim}_{h\to 0}\mathcal{E}^{M_p,h}_{K}(X),
$ and
\begin{align*}
\mathcal{E}^{M_p,h}_{K}(X):=\Biggl\{\phi\in C^\infty(\mathbb{R} :X): \sup_{p\geq 0}\frac{h^{p}\|\phi^{(p)}\|_{C(K : X)}}{M_{p}}<\infty\Biggr\}.
\end{align*}
The space consisted of all linear continuous mappings $\mathcal{E}^*({\mathbb C})\rightarrow X$ is denoted by $\mathcal{E}^{\prime *}(X);$
$\mathcal{E}^{\prime *}:=\mathcal{E}^{\prime *}({\mathbb C}).$
Let us recall \cite{k91} that an entire function of the form
$P(\lambda)=\sum_{p=0}^{\infty}a_p\lambda^p$,
$\lambda\in\mathbb{C}$, is of class $(M_p)$, resp., of
class $\{M_p\}$, if there exist $l>0$ and $C>0$, resp., for every
$l>0$ there exists a constant $C>0$, such that $|a_p|\leq Cl^p/M_p$,
$p\in\mathbb{N}$.
The corresponding ultradifferential operator
$P(D)=\sum_{p=0}^{\infty}a_p D^p$ is of class $(M_p)$, resp., of
class $\{M_p\}$.
The convolution of Banach space valued
ultradistributions and scalar-valued ultradifferentiable functions of the same class will be taken in the sense of considerations given on page 685 of \cite{k82}. We have that, for every
$f\in \mathcal{D}^{\prime *}(X)$ and $\varphi \in \mathcal{D}^*,$ $f\ast \varphi \in \mathcal{E}^*(X)$ as well as that the linear mapping $\varphi \mapsto \cdot \ast \varphi : \mathcal{D}^{\prime *}(X) \rightarrow \mathcal{E}^*(X)$ is continuous. The convolution of an $X$-valued
ultradistribution $f(\cdot)$ and an element $g\in {\mathcal E}^{\prime \ast },$ defined by the identity \cite[(4.9)]{k82}, is an $X$-valued
ultradistribution and the mapping $g \ast \cdot :  \mathcal{D}^{\prime *}(X) \rightarrow \mathcal{D}^{\prime *}(X)$ is continuous. Put $\langle T_{h},\varphi \rangle :=\langle T, \varphi (\cdot -h) \rangle,$ $T\in \mathcal{D}^{\prime *}(X),$ $\varphi \in {\mathcal D}^{\ast}$ ($h>0$).

If $(M_{p})$ satisfies (M.1), (M.2) and (M.3), then
$$
P_{l}(x)=\bigl( 1+x^{2} \bigr)\prod_{p\in {\mathbb N}}\Biggl(1+\frac{x^{2}}{l^{2}m_{p}^{2}}\Biggr),
$$
resp.
$$
P_{r_{p}}(x)=\bigl( 1+x^{2} \bigr)\prod_{p\in {\mathbb N}}\Biggl(1+\frac{x^{2}}{r_{p}^{2}m_{p}^{2}}\Biggr),
$$
defines an ultradifferential operator of class $(M_p)$, resp., of
class $\{M_p\}$. Here, $(r_{p})$ is a sequence of positive real numbers tending to infinity. The family consisting of all such sequences will be denoted by ${\mathrm R}$ henceforth. For more details, see \cite{k91}-\cite{k82}.

The spaces of tempered ultradistributions of Beurling,
resp., Roumieu type, are defined by S. Pilipovi\' c \cite{pilip} as duals of the corresponding test spaces
$$
\mathcal{S}^{(M_p)}:=\text{projlim}_{h\to\infty}\mathcal{S}^{M_p,h},\
\mbox{ resp., }\ \mathcal{S}^{\{M_p\}}:=\text{indlim}_{h\to 0}\mathcal{S}^{M_p,h},
$$
where
\begin{gather*}
\mathcal{S}^{M_p,h}:=\bigl\{\phi\in C^\infty(\mathbb{R}):\|\phi\|_{M_p,h}<\infty\bigr\}\ \ (h>0),
\\
\|\phi\|_{M_p,h}:=\sup\Biggl\{\frac{h^{\alpha+\beta}}{M_\alpha M_\beta}(1+t^2)^{\beta/2}|\phi^{(\alpha)}(t)|:t\in\mathbb{R},
\;\alpha,\;\beta\in\mathbb{N}_0\Biggr\}.
\end{gather*}
A continuous linear mapping $
\mathcal{S}^{(M_p)} \rightarrow X,$ resp., $
\mathcal{S}^{\{M_p\}}\rightarrow X,$ is said to be an $X$-valued tempered ultradistribution of Beurling, resp., Roumieu type. The space consisting of all vector-valued tempered ultradistributions of Beurling, resp., Roumieu type,
will be denoted by $\mathcal{S}^{\prime (M_p)}(X),$ resp. $\mathcal{S}^{\prime \{M_p\} }(X);$ the common shorthand will be $\mathcal{S}^{\prime \ast}(X).$ It is well known that
$\mathcal{S}^{\prime (M_p)}(X) \subseteq \mathcal{D}^{\prime (M_p)}(X)$, resp. $\mathcal{S}^{\prime\{M_p\} }(X) \subseteq \mathcal{D}^{\prime\{M_p\} }(X) .$

\section{Almost automorphy of vector-valued distributions}\label{collado-automorphic}

We will use the following elementary notion (see L. Schwartz \cite{sch16} for more details). The symbol ${\mathcal D}={\mathcal D}({\mathbb R})$ denotes the Schwartz space of test functions, the space of rapidly decreasing functions ${\mathcal S}={\mathcal S}({\mathbb R})$ carries the usual Fr\' echet topology and ${\mathcal E}={\mathcal E}({\mathbb R}),$ the space of all infinitely differentiable functions, carries with the usual Fr\' echet topology. By
${\mathcal D}^{\prime}(X),$ ${\mathcal S}^{\prime}(X)$ and ${\mathcal E}^{\prime}(X)$ we denote the spaces of all linear continuous mappings ${\mathcal D} \rightarrow X,$ ${\mathcal S} \rightarrow X$ and ${\mathcal E} \rightarrow X,$ respectively.

Our first task in this section will be to verify that all structural results proved by C. Bouzar and F. Z. Tchouar \cite{buzar-tress} continue to hold in vector-valued case.
Let $1\leq p \leq \infty$. Then ${\mathcal D}_{L^{p}}(X)$ denote the vector space consisting of all infinitely differentiable functions $f: {\mathbb R} \rightarrow X$ satisfying that for each number $j\in {\mathbb N}_{0}$ we have $f^{(j)}\in L^{p}({\mathbb R} : X).$ The Fr\' echet topology on ${\mathcal D}_{L^{p}}(X)$ is induced by the following system of seminorms
$$
\|f\|_{k}:=\sum_{j=0}^{k}\bigl\|f^{(j)}\bigr\|_{L^{p}({\mathbb R})},\quad f\in {\mathcal D}_{L^{p} (X) } \ \ \bigl( k\in {\mathbb N}\bigr).
$$
In the case that $X={\mathbb C},$ then the above space is simply denoted by ${\mathcal D}_{L^{p}}.$
A continuous linear mapping $f : {\mathcal D}_{L^{1}} \rightarrow X$ is said to be a bounded $X$-valued distribution; the space consisting of such vector-valued distributions will be denoted by ${\mathcal D}^{\prime}_{L^{1}}(X)$. By $B^{\prime}(X)$ we denote the space of such distributions; endowed with the strong topology, $B^{\prime}(X)$ becomes a complete locally convex space. For every $f\in B^{\prime}(X)$, we have that $f_{| {\mathcal S}} : {\mathcal S} \rightarrow X$ is a tempered $X$-valued distribution.

Set
$$
\mathcal{E}_{AA}(X):=\Biggl\{ \phi \in \mathcal{E}(X) : \phi^{(i)} \in AA({\mathbb R} : X)\mbox{ for all }i\in {{\mathbb N}_{0}}\Biggr\}.
$$
Since, for every $\phi \in \mathcal{E}_{AA}(X),$ we have $\phi \in AA({\mathbb R} : X)\subseteq C_{b}({\mathbb R} : X)$ and
$$
\Biggl\|\int^{\infty}_{-\infty}\phi(t) \varphi(t)dt\Biggr\|\leq \|\phi\|_{L^{\infty}({\mathbb R})}\|\varphi\|_{L^{1}},\quad \varphi \in {\mathcal D}_{L^{1}},
$$
the mapping $\phi \mapsto \int^{\infty}_{-\infty}\phi(t) \varphi(t)dt,$ $\varphi \in {\mathcal D}_{L^{1}}$ is linear and continuous so that
$\mathcal{E}_{AA}(X) \subseteq {\mathcal D}^{ \prime  }_{L^{1}}( X).$ Using the fact that the first derivative of a differentiable almost automorphic function is almost automorphic iff it is uniformly continuous \cite{diagana}, it can be easily verified that we have $\mathcal{E}_{AA}(X)=\mathcal{E}(X) \cap AA({\mathbb R} : X);$ furthermore, $\mathcal{E}_{AA}(X) \ast L^{1}({\mathbb R}) \subseteq \mathcal{E}_{AA}(X)$
and $\mathcal{E}_{AA}(X)$ is a closed subspace of ${\mathcal D}_{L^{\infty}}(X)$ (see \cite[Proposition 5]{buzar-tress}). We have, actually, that $\mathcal{E}_{AA}(X)$ is the space of those elements $f(\cdot)$ from ${\mathcal D}_{L^{\infty}}(X)$
for which $f\ast \varphi \in AA({\mathbb R} : X),$ $\varphi \in {\mathcal D}$  see \cite[Corollary 1]{buzar-tress}. For any vector-valued distribution $T\in {\mathcal D}^{\prime}(X),$ we define $\tau_{h}T:=T_{h}$ by $\langle T_{h} ,\varphi \rangle :=\langle T, \varphi(\cdot -h)\rangle,$ $\varphi \in {\mathcal D}$
($h \in {\mathbb R}$).

The following result is crucial:

\begin{thm}\label{kruc-auto} (see \cite[Theorem 1]{buzar-tress})
Let $T\in {\mathcal D}^{\prime}_{L^{1}}(X).$ Then the following assertions are equivalent:
\begin{itemize}
\item[\emph{(i)}] $T \ast \varphi \in AA({\mathbb R} : X),$ $\varphi \in {\mathcal D}.$
\item[\emph{(ii)}] There exist an integer $k\in {\mathbb N}$ and almost automorphic functions $f_{j}(\cdot) : {\mathbb R} \rightarrow X$ ($1\leq j\leq k$) such that $T=\sum_{j=0}^{k}f_{j}^{(j)}.$
\end{itemize}
\end{thm}

It is said that a distribution $T\in {\mathcal D}^{\prime}_{L^{1}}(X)$ is almost automorphic iff $T$ satisfies any of the above two equivalent conditions. By $B^{\prime}_{AA}(X)$ we denote the space consisting of all almost automorphic distributions. The space  $B^{\prime}_{AA}(X)$ is closed under differentiation and \cite[Proposition 6]{buzar-tress} continue to hold in vector-valued case. This is also the case with the assertions of
\cite[Proposition 7, Proposition 8, Theorem 2, Proposition 9, Proposition 10]{buzar-tress}, so that we have the following theorem:

\begin{thm}\label{prepisivanje}
\begin{itemize}
\item[\emph{(i)}] Let $T\in {\mathcal D}^{\prime}_{L^{1}}(X).$ Then $T$ is almost automorphic iff for every real sequence $(b_{n}),$ there exist a subsequence $(a_{n})$ of $(b_{n})$ and a vector-valued distribution $S \in {\mathcal D}^{\prime}(X)$ such that $\lim_{n\rightarrow \infty}T_{a_{n}}=S$ in ${\mathcal D}^{\prime}(X)$ and $\lim_{n\rightarrow \infty}S_{-a_{n}}=T$ in ${\mathcal D}^{\prime}(X)$ iff there exists a sequence of almost automorphic functions converging to $T$ in ${\mathcal D}^{\prime}_{L^{1}}(X)$ iff for every real sequence $(b_{n}),$ there exists a subsequence $(a_{n})$ of $(b_{n})$ such that $\lim_{l\rightarrow \infty}\lim_{k\rightarrow \infty}\tau_{-a_{l}}\tau_{a_{n}}T=T$
in ${\mathcal D}^{\prime}(X).$
\item[\emph{(ii)}] Let $f\in  AAS^{p}({\mathbb R} : X)$  for some $p\in [1,\infty).$ Then the regular distribution associated to $f(\cdot)$ is almost automorphic.
\end{itemize}
\end{thm}

\section{Almost automorphy of vector-valued ultradistributions}\label{collado-automorphic-ultra}

For any $h>0,$ we set
$$
{\mathcal D}_{L^{1}}\bigl((M_{p}),h\bigr):=\Biggl\{ f\in {\mathcal D}_{L^{1}} \ ; \ \|f\|_{1,h}:=\sup_{p\in {\mathbb N}_{0}}\frac{h^{p}\|f^{(p)}\|_{1}}{M_{p}}<\infty \Biggr\} .
$$
Then $({\mathcal D}_{L^{1}}((M_{p}),h),\| \cdot \|_{1,h})$ is a Banach space and the space of all $X$-valued bounded Beurling ultradistributions of class $(M_{p})$, resp., $X$-valued bounded Roumieu ultradistributions of class $\{M_{p}\}$,
is defined as the space consisting of all linear continuous mappings from ${\mathcal D}_{L^{1}}((M_{p})),$ resp., $
{\mathcal D}_{L^{1}}(\{M_{p}\}),$ into $X,$ where
$$
{\mathcal D}_{L^{1}}\bigl((M_{p})\bigr):=\text{projlim}_{h\rightarrow +\infty}{\mathcal D}_{L^{1}}\bigl((M_{p}),h\bigr),
$$
resp.,
$$
{\mathcal D}_{L^{1}}\bigl(\{M_{p}\}\bigr):=\text{indlim}_{h\rightarrow 0+}{\mathcal D}_{L^{1}}\bigl((M_{p}),h\bigr).
$$
These spaces, equppied with the strong topologies, will be shortly denoted by $
{\mathcal D}_{L^{1}}^{ \prime }((M_{p}) : X),$ resp., $
{\mathcal D}_{L^{1}}^{ \prime }(\{M_{p}\} : X).$ It is well known that ${\mathcal D}^{(M_{p})},$ resp. ${\mathcal D}^{\{M_{p}\}},$ is a dense subspace of
${\mathcal D}_{L^{1}}((M_{p}) ),$ resp., $
{\mathcal D}_{L^{1}}(\{M_{p}\} ),$ as well as that  ${\mathcal D}_{L^{1}}((M_{p}) )\subseteq
{\mathcal D}_{L^{1}}(\{M_{p}\} )$. Since $\|\varphi\|_{1,h}\leq \|\varphi\|_{M_p,h}$ for any $\varphi \in \mathcal{S}^{(M_p)}$ and $h>0,$ we have that
${\mathcal S}^{(M_{p})},$ resp. ${\mathcal S}^{\{M_{p}\}},$ is a dense subspace of
${\mathcal D}_{L^{1}}((M_{p}) ),$ resp., $
{\mathcal D}_{L^{1}}(\{M_{p}\} ),$ and that $f_{| {\mathcal S}^{(M_{p})}} : {\mathcal S}^{(M_{p})} \rightarrow X,$ resp., $f_{| {\mathcal S}^{\{M_{p}\}}} : {\mathcal S}^{\{M_{p}\}} \rightarrow X,$ is a tempered $X$-valued ultradistribution of class $(M_{p}),$ resp.,
of class $\{M_{p}\}.$ The space $
{\mathcal D}_{L^{1}}^{ \prime }((M_{p}) : X),$ resp. $
{\mathcal D}_{L^{1}}^{ \prime }(\{M_{p}\} : X),$ is closed under the action of ultradifferential operators of $(M_{p})$-class, resp. $\{M_{p}\}$-class.

Assume that ${\mathrm A}\subseteq {\mathcal D}^{\prime \ast}(X).$ Following the investigation of B. Basit and H. G\"uenzler \cite{basit-duo-gue}, conducted for vector-valued distributions, we have recently introduced the following notion in \cite{mat-bilten}:
$$
\mathcal{D}^{\prime *}_{{\mathrm A}}(X):=\Bigl\{  T\in\ \mathcal{D}^{\prime *}(X) : T\ast \varphi \in {\mathrm A}\mbox{ for all }\varphi \in \mathcal{D}^{*}\Bigr\}.
$$
It is worth noting that $
\mathcal{D}^{\prime *}_{{\mathrm A}}(X)=
\mathcal{D}^{\prime *}_{{\mathbb A}\cap B}(X),$ for any set $B\subseteq L_{loc}^{1}({\mathbb R} : X)$ that contains $C^{\infty}({\mathbb R} : X),$ as well as that the set $
\mathcal{D}^{\prime *}_{{\mathrm A}}(X)$ is closed under the action of ultradifferential operators of $\ast$-class. In \cite{mat-bilten}, we have proved the following assertions:

\begin{itemize}
\item[(i)]
Suppose that there exist an ultradifferential operator $P(D)=\sum_{p=0}^{\infty}a_p D^p$
of class $(M_p)$, resp., of
class $\{M_p\},$ and $f,\ g\in \mathcal{D}^{\prime *}_{{\mathrm A}}(X)$ such that $T=P(D)f+g.$ If ${\mathrm A}$ is closed under addition, then
$T\in \mathcal{D}^{\prime *}_{{\mathrm A}}(X).$
\item[(ii)] If ${\mathrm A} \cap C({\mathbb R} : X)$ is closed under uniform convergence, $T\in
{\mathcal D}_{L^{1}}^{ \prime }((M_{p}) : X)$ and $T\ast \varphi \in {\mathrm A},$ $\varphi \in {\mathcal D}^{(M_{p})},$ then there exists $h>0$ such that for each compact set $K\subseteq {\mathbb R}$ we have $T\ast \varphi \in {\mathrm A},$ $\varphi \in \mathcal{D}^{M_p,h}_K.$
\item[(iii)] Suppose that $T\in \mathcal{D}^{\prime (M_{p})}(X)$ and there exists $h>0$ such that for each compact set $K\subseteq {\mathbb R}$ we have $T\ast \varphi \in {\mathrm A},$ $\varphi \in \mathcal{D}^{M_p,h}_K.$ If  $(M_p)$ additionally satisfies (M.3), then there exist $l>0$ and two elements
$f,\ g\in {\mathrm A}$ such that $T=P(D)f+g.$
\end{itemize}

Now we will consider the case that ${\mathrm A}=AA({\mathbb R} : X).$ Then ${\mathrm A}$ is closed under the uniform convergence and addition, and we have ${\mathrm A}\subseteq \mathcal{D}^{\prime *}_{{\mathrm A}}(X)$ (\cite{nova-mono}). Hence, as a special case of the above assertions, we have the following result:

\begin{thm}\label{mare-basit}
Let $(M_{p})$
satisfy the conditions \emph{(M.1), (M.2)} and \emph{(M.3'),} and let $T\in
{\mathcal D}^{ \prime (M_{p})}(X),$ resp., $T\in
{\mathcal D}^{ \prime \{M_{p}\}}(X).$ Then the following holds:
\begin{itemize}
\item[\emph{(i)}]
Suppose that there exist an ultradifferential operator $P(D)=\sum_{p=0}^{\infty}a_p D^p$
of class $(M_p)$, resp., of
class $\{M_p\},$ and $f,\ g\in \mathcal{D}^{\prime *}_{AA({\mathbb R} : X)}(X)$ such that $T=P(D)f+g.$ Then
$T\in \mathcal{D}^{\prime *}_{AA({\mathbb R} : X)}(X).$
\item[\emph{(ii)}] If $T\in
{\mathcal D}_{L^{1}}^{ \prime }((M_{p}) : X)$ and $T\ast \varphi \in AA({\mathbb R} : X),$ $\varphi \in {\mathcal D}^{(M_{p})},$ then there exists $h>0$ such that for each compact set $K\subseteq {\mathbb R}$ we have $T\ast \varphi \in AA({\mathbb R} : X),$ $\varphi \in \mathcal{D}^{M_p,h}_K.$
\item[\emph{(iii)}] Suppose that $T\in \mathcal{D}^{\prime (M_{p})}(X)$ and there exists $h>0$ such that for each compact set $K\subseteq {\mathbb R}$ we have $T\ast \varphi \in AA({\mathbb R} : X),$ $\varphi \in \mathcal{D}^{M_p,h}_K.$ If  $(M_p)$ additionally satisfies \emph{(M.3),} then there exist $l>0$ and two elements
$f,\ g\in AA({\mathbb R} : X)$ such that $T=P(D)f+g.$
\end{itemize}
\end{thm}

As an immediate corollary of Theorem \ref{mare-basit}, we have the following:

\begin{cor}\label{kostici-bodovi}
Let $(M_{p})$
satisfy the conditions \emph{(M.1), (M.2)} and \emph{(M.3)',} and let $T\in
{\mathcal D}_{L^{1}}^{ \prime }((M_{p}) : X),$ resp. $T\in
{\mathcal D}_{L^{1}}^{ \prime }(\{M_{p}\} : X)$. Consider now the following assertions:
\begin{itemize}
\item[\emph{(i)}] There exist a number $l>0,$ resp. a sequence $(r_{p})\in {\mathrm R},$ and two functions $f,\ g\in AA({\mathbb R} : X)$  such that $T=P_{l}(D)f+g,$ resp.  $T=P_{r_{p}}(D)f+g.$
\item[\emph{(ii)}] There exist an ultradifferential operator $P(D)=\sum_{p=0}^{\infty}a_p D^p$
of class $(M_p),$ resp. $\{M_{p}\}$, and two functions $f,\ g\in AA({\mathbb R} : X)$ such that $T=P(D)f+g.$
\item[\emph{(iii)}] We have $T\ast \varphi \in AA({\mathbb R} : X),$ $\varphi \in {\mathcal D}^{\ast}.$
\item[\emph{(iv)}] There exists $h>0$ such that for each compact set $K\subseteq {\mathbb R},$ resp. for each $h>0$ and for each compact set $K\subseteq {\mathbb R},$ we have $T\ast \varphi \in AA({\mathbb R} : X),$ $\varphi \in \mathcal{D}^{M_p,h}_K.$
\end{itemize}
Then we have \emph{(i)} $\Rightarrow$ \emph{(ii)} $\Rightarrow$ \emph{(iii)} $\Leftrightarrow$ \emph{(iv)}. Furthermore, if $(M_{p})$
additionally satisfies the condition \emph{(M.3),} then the assertions \emph{(i)}-\emph{(iv)} are equivalent for the Beurling class.
\end{cor}

Let us introduce the following space
$$
\mathcal{E}^{*}_{AA}(X):=\Biggl\{ \phi \in \mathcal{E}^*(X) : \phi^{(i)} \in AA({\mathbb R} : X)\mbox{ for all }i\in {{\mathbb N}_{0}}\Biggr\}.
$$
As in distribution case,
$\mathcal{E}^{*}_{AA}(X) \subseteq {\mathcal D}_{L^{1}}^{ \prime \ast }( X),$ $\mathcal{E}^{*}_{AA}(X)=\mathcal{E}^{*}(X) \cap AA({\mathbb R} : X)$ and $\mathcal{E}^{*}_{AA}(X) \ast L^{1}({\mathbb R}) \subseteq \mathcal{E}^{*}_{AA}(X);$ furthermore, $\mathcal{E}^{*}_{AA}(X)$ is the space of those elements $f(\cdot)$ from $\mathcal{E}^*(X)$
for which $f\ast \varphi \in AA({\mathbb R} : X),$ $\varphi \in {\mathcal D}^{*}.$

Consider now the following statement:
\begin{itemize}
\item[(ii)':] $T\in
{\mathcal D}_{L^{1}}^{ \prime }((M_{p}) : X),$ resp. $T\in
{\mathcal D}_{L^{1}}^{ \prime }(\{M_{p}\} : X),$ and there exists a sequence $(\phi_{n})$ in $\mathcal{E}^{*}_{AA}(X)$ such that
$\lim_{n\rightarrow \infty}\phi_{n}=T$ for the topology of ${\mathcal D}_{L^{1}}^{ \prime }((M_{p}) : X),$ resp. $
{\mathcal D}_{L^{1}}^{ \prime }(\{M_{p}\} : X).$
\end{itemize}

The proof of following proposition is almost the same as that of \cite[Lemma 1]{mat-bilten}:

\begin{prop}\label{tom-waits-automorphic}
Let $(M_{p})$
satisfy the conditions \emph{(M.1), (M.2)} and \emph{(M.3'),} and let $T\in
{\mathcal D}_{L^{1}}^{ \prime }((M_{p}) : X),$ resp.  $T\in
{\mathcal D}_{L^{1}}^{ \prime }(\{M_{p}\} : X)$. Then we have \emph{(iii)} $\Leftrightarrow$ \emph{(ii)'},
with \emph{(iii)} being the same as in the formulation of \emph{Corollary \ref{kostici-bodovi}}.
\end{prop}

It is said that a bounded ultradistribution $T\in
{\mathcal D}_{L^{1}}^{ \prime }((M_{p}) : X),$ resp. $T\in
{\mathcal D}_{L^{1}}^{ \prime }(\{M_{p}\} : X),$ is almost automorphic iff $T$ satisfies any of the above two equivalent conditions. It can be simply verified that a regular distribution (ultradistribution of $\ast$-class) determined by an almost automorphic vector-valued function that is not almost periodic is an  almost automorphic vector-valued distribution (ultradistribution of $\ast$-class) that cannot be almost periodic (cf. \cite[Example 2]{buzar-tress}).

Now we would like to state the following result:

\begin{thm}\label{ultra-automorphic-ras}
Let $(M_{p})$
satisfy the conditions \emph{(M.1), (M.2)} and \emph{(M.3'),} and let $T\in
{\mathcal D}_{L^{1}}^{ \prime }((M_{p}) : X),$ resp.  $T\in
{\mathcal D}_{L^{1}}^{ \prime }(\{M_{p}\} : X)$. Then we have \emph{(i)} $\Rightarrow$ \emph{(ii)} $\Rightarrow$ \emph{(iii)} $\Rightarrow$ \emph{(iv)}, where:
\begin{itemize}
\item[\emph{(i)}] There exist an ultradifferential operator $P(D)=\sum_{p=0}^{\infty}a_p D^p$
of class $(M_p),$ resp. $\{M_{p}\}$, and two functions $f,\ g\in AA({\mathbb R} : X)$ such that $T=P(D)f+g.$
\item[\emph{(ii)}] For every real sequence $(b_{n}),$ there exist a subsequence $(a_{n})$ of $(b_{n})$ and a vector-valued ultradistribution $S \in {\mathcal D}^{\prime \ast}(X)$ such that $\lim_{n\rightarrow \infty}\langle T_{a_{n}},\varphi \rangle=\langle S,\varphi \rangle,$ $\varphi \in {\mathcal D}^{\ast}$ and $\lim_{n\rightarrow \infty}\langle S_{-a_{n}},\varphi \rangle=\langle T, \varphi \rangle,$ $\varphi \in {\mathcal D}^{\ast}.$
\item[\emph{(iii)}] For every real sequence $(b_{n}),$ there exists a subsequence $(a_{n})$ of $(b_{n})$ such that $\lim_{l\rightarrow \infty}\lim_{k\rightarrow \infty}\langle \tau_{-a_{l}}\tau_{a_{k}}T , \varphi \rangle=\langle T, \varphi \rangle,$ $\varphi \in {\mathcal D}^{\ast}.$
\item[\emph{(iv)}] We have $T\ast \varphi \in AA({\mathbb R} : X),$ $\varphi \in {\mathcal D}^{\ast}.$
\end{itemize}
Furthermore, if $(M_{p})$
additionally satisfies the condition \emph{(M.3),} then the assertions \emph{(i)}-\emph{(iv)} are equivalent for the Beurling class.
\end{thm}

\begin{proof}
The proof of (ii) $\Rightarrow$ (iii) $\Rightarrow$ (iv) can be deduced as in distribution case (see e.g. the proof of \cite[Proposition 9]{buzar-tress}). For the proof of implication (i) $\Rightarrow$ (ii), observe that the almost automorphy of functions $f(\cdot)$ and $g(\cdot)$ implies the existence of essentially bounded functions $F\in L^{\infty}({\mathbb R} : X)$ and  $G\in L^{\infty}({\mathbb R} : X)$ such that
\begin{align*}
\lim_{n\rightarrow \infty}f\bigl( t+a_{n}\bigr)=F(t)\ \mbox{ and } \  \lim_{n\rightarrow \infty}F\bigl( t-a_{n}\bigr)=f(t)
\end{align*}
and
\begin{align*}
\lim_{n\rightarrow \infty}g\bigl( t+a_{n}\bigr)=G(t)\ \mbox{ and } \  \lim_{n\rightarrow \infty}G\bigl( t-a_{n}\bigr)=g(t),
\end{align*}
pointwise for $t\in {\mathbb R}.$ Using these equations, the dominated convergence theorem and the fact that, for every bounded subset $B$ of ${\mathcal D}^{\ast}$ and for every compact set $K\subseteq {\mathbb R},$ there exists $h>0$ such that $B$ is bounded in $\mathcal{D}^{M_p,h}_K$ (\cite{k91}), it readily follows that $\lim_{n\rightarrow \infty}f(\cdot +a_{n})=F$ in ${\mathcal D}^{\prime \ast}(X)$
and $\lim_{n\rightarrow \infty}g(\cdot +a_{n})=G$ in ${\mathcal D}^{\prime \ast}(X),$ so that $\lim_{n\rightarrow \infty}\langle T_{a_{n}},\varphi \rangle = \langle S,\varphi \rangle,$ where $S\in {\mathcal D}^{\prime \ast}(X)$ is given by $S:=P(D)F+G.$ Similarly we can deduce that $\lim_{n\rightarrow \infty}\langle S_{-a_{n}},\varphi \rangle=\langle T, \varphi \rangle,$ $\varphi \in {\mathcal D}^{\ast},$ finishing the proof of theorem.
\end{proof}

In the present situation, we cannot tell whether the implication (iv) $\Rightarrow$ (ii) holds true in general case.

In \cite[Section 6]{buzar-tress}, C. Bouzar and F. Z. Tchouar have continued the analysis of S. Bochner \cite{bochner} concerning linear difference-differential operators
$$
L_{h}=\sum_{i=0}^{p}\sum_{j=0}^{q}a_{ij}\frac{d^{i}}{dx^{i}}\tau_{h_{j}},
$$
where $a_{ij}$ are complex numbers ($0\leq i\leq p,$ $0\leq j\leq q$) and $h=(h_{j})_{0\leq j\leq q}\subseteq {\mathbb R}^{q}$. Taking into account the fact that \cite[Theorem 4(i)]{bochner} holds in vector-valued case, \cite[Theorem 1, Theorem 3]{mat-bilten} and the proof of \cite[Theorem 3]{buzar-tress}, we can simply clarify that the assertions of \cite[Theorem 3, Corollary 3]{buzar-tress} hold for vector-valued distributions, as well as for vector-valued ultradistributions:

\begin{thm}\label{baobab}
Let $C^{p}_{buc}({\mathbb R} :X)$ denote the vector space of all $p$-times differentiable uniformly continuous functions $f\in BUC({\mathbb R} : X)$ for which $f^{(j)}\in BUC({\mathbb R} : X),$ $0\leq j\leq p.$ Let $S\in
{\mathcal D}_{L^{1}}^{ \prime }(\{M_{p}\} : X),$ resp. $S\in
{\mathcal D}_{L^{1}}^{ \prime }(\{M_{p}\} : X),$ be almost automorphic.
\begin{itemize}
\item[\emph{(i)}] If every solution $f\in C^{p}_{buc}({\mathbb R} :X)$ of the homogeneous equation $L_{h}f=0$ is almost automorphic, then every solution $T\in
{\mathcal D}_{L^{1}}^{ \prime }((M_{p}) : X),$ resp. $T\in
{\mathcal D}_{L^{1}}^{ \prime }(\{M_{p}\} : X),$ of the
inhomogeneous equation $L_{h}T=S$ is almost automorphic.
\item[\emph{(ii)}] If $S^{\prime}$ is almost automorphic, then $S$ is almost automorphic.
\item[\emph{(iii)}] Any translation $S_{h}$ of $S$ is almost automorphic ($h\in {\mathbb R}$).
\end{itemize}
\end{thm}

We close the paper with the observation that the assertions of \cite[Theorem 4, Corollary 4]{buzar-tress} can be also formulated for vector-valued distributions and vector-valued ultradistributions.

\section*{Acknowledgement}
The authors are partially supported by
grant 174024 of Ministry of Science and Technological Development,
Republic of Serbia.

\end{document}